\documentclass[10pt]{amsart}
\usepackage{amsmath,amsfonts,latexsym,amssymb,amscd,
  graphicx,amsthm, mathtools}
\usepackage{tikz, subcaption}
\usepackage{enumerate}
\usepackage{ulem}
\usepackage{cancel}
\usepackage{hyperref}

\allowdisplaybreaks[1]

\newcommand{\ONE}{{\mathbf{1}}}

\newcommand{\N}{{\mathbb N}}

\newcommand{\Fc}{\mathcal{F}}

\newcommand{\Pp}{\mathsf{P}}
\newcommand{\Z}{{\mathbb Z}}
\newcommand{\iid}{i.i.d.\ }

\newcommand{\Prm}{{\mathrm P}}
\newcommand{\R}{{\mathbb R}}

\newcommand{\Bc}{{\mathcal B}}

\newtheorem{theorem}{Theorem}[section]

\newtheorem{lemma}{Lemma}[section]
\newtheorem{corollary}{Corollary}[section]
\newtheorem{definition}{Definition}[section]

\numberwithin{equation}{section}

\renewenvironment{proof}[1][Proof]{% beginning of proof
{\noindent {\sc #1: }}
}{% end of proof
{{\hfill $\Box$ \smallskip}}
}
\makeatletter
\let\orgdescriptionlabel\descriptionlabel
\renewcommand*{\descriptionlabel}[1]{%
  \let\orglabel\label
  \let\label\@gobble
  \phantomsection
  \edef\@currentlabel{#1}%
  \let\label\orglabel
  \orgdescriptionlabel{#1}%
}
\makeatother

\title[]{Weakly mixing smooth planar vector
  field without asymptotic directions}
\author{Yuri Bakhtin}
\address{Courant Institute of Mathematical Sciences\\ New York University \\ 251 Mercer St, New York, NY 10012 }
\author{Liying Li}
\email{bakhtin@cims.nyu.edu, liying@cims.nyu.edu}
\begin{document}

\begin{abstract}
We construct a planar smooth weakly mixing stationary random  vector field with
nonnegative components such that, with probability 1, the flow generated by this vector field does not have an asymptotic direction. Moreover,  for all individual trajectories, the set of partial limiting directions coincides with those spanning the positive quadrant. A modified example shows that a particle in space-time weakly mixing positive velocity field does not necessarily have an asymptotic average velocity.
\end{abstract}

 \maketitle

\section{Introduction and the main results}\label{sec:motivation} 

Homogenization problems for stochastic Hamilton--Jacobi (HJ) type equations (see~\cite{Souganidis:MR1697831},\cite{Rezakhanlou-Tarver:MR1756906},\cite{Nolen-Novikov:MR2815685},\cite{Cardaliaguet-Souganidis:MR3084699},\cite{Jing-Souganidis-Tran:MR3817561}) and limit shape problems in First Passage Percolation (FPP) type models (see a recent book \cite{AHDbook:MR3729447}) are tightly related to asymptotic properties of optimal paths in random environments.  In several interesting situations where the setup involves stationarity and fast decorrelation of the environment, one can prove that optimal paths solving the control problem in the variational characterization of solutions in the HJ case and the random geodesics in the FPP case have some kind of straightness property (see \cite{Licea1996},\cite{HoNe},\cite{Wu},\cite{CaPi},\cite{Cardaliaguet-Souganidis:MR3084699},\cite{BCK:MR3110798},\cite{kickb:bakhtin2016}). In particular, for a one-sided semi-infinite minimizer or geodesic $\gamma$, existence of a well-defined asymptotic direction $\lim_{t\to\infty} (\gamma(t)/t)$ has been established for several models. 
The FPP limit shape and the effective Lagrangian (aka shape function) in stationary control problems are always convex.
In the literature cited above, it is shown that stronger assumptions on curvature imply quantitative estimates on deviations from straightness, and it is believed that the asymptotic behavior of these deviations are often governed by KPZ scalings (see, e.g., \cite{Bakhtin-Khanin-non:MR3816628}).

It is tempting to conjecture that in a closely related and simpler passive tracer setting where the control is eliminated and the particles simply flow along an ergodic stationary random vector field,
each of the resulting trajectories will have an asymptotic direction. However, the generality of this picture is limited, and the main goal of
this note is to construct a weakly mixing stationary random field $v$ on $\R^2$ such that none of its integral curves possesses a limiting direction.  
We refer to Section~\ref{sec:appendix} for reminders on weak mixing and related notions and recall here only that weak mixing is stronger than ergodicity.

The construction in our main result is based on lifting the discrete $\Z^2$-ergodic example recently
introduced in~\cite{2016arXiv161200434C} onto $\R^2$ with the help of appropriate tilings, smoothing, and additional randomizations. In fact, the Poissonization that we use can also be used to construct a $\Z^2$-weak mixing example out of the ergodic example of~\cite{2016arXiv161200434C}.

Our vector field (along with the discrete arrow field of~\cite{2016arXiv161200434C}) traps the integral curves in long corridors each stretched along one of the two prescribed extreme directions, and the random length of these corridors has heavy tails. The result is that the integral curves oscillate between these two directions never settling on any specific one. This idea is similar to that of~\cite{Haggstrom-Meester:MR1379157}, where it is shown how to construct an FPP model with any given convex symmetric limit shape, by employing long random properly directed corridors that are easy to percolate along. As noted in~\cite{2016arXiv161200434C}, the absence of a well-defined average velocity is a manifestation of the fact that there is no averaging of the environment as seen from the particle moving along the random realization of the vector field.

\medskip
 
Let us be more precise now. For every bounded smooth vector field $v$ on $\R^2$ and every initial condition
$z\in\R^2$, we can define the integral 
curve~$\gamma_z: \R_+  \to \R^2$  (here $\R_+=[0,\infty)$) as a unique solution of the autonomous ODE
\begin{equation}
\label{eq:ode}
\dot \gamma_z(t) = v\big( \gamma_z(t) \big),
\end{equation}
satisfying
\begin{equation}
\label{eq:init-cond-ode}
\gamma_z(0) = z.
\end{equation}
We denote the two components of $v\in\R^2$ by  $v^1$ and $v^2$.
\begin{theorem}\label{th:no-average-slope}
There is a weakly mixing stationary random vector field $v\in C^{\infty}(\R^2)$ such that
for all $z\in\R$,  $v^1(z),v^2(z)\ge 0$, $v^1(z)+v^2(z)>0$, and with probability~$1$ the following holds for all
$z\in\R^2$:
\begin{equation}
\label{eq:trajects-go-to-infty}
\lim_{t\to \infty}|\gamma_z(t)|=\infty,
\end{equation}
 \begin{equation}\label{eq:integral-curves-no-direction}
  \liminf_{t \to \infty} \frac{\gamma^{2}_z(t)}{  \gamma^{1}_z (t) } = 0, \quad
  \limsup_{t \to \infty} \frac{\gamma^{2}_z(t) }{  \gamma^{1}_z (t)} = \infty.
\end{equation}
\end{theorem}
In other words, the random vector field that we construct in this theorem generates integral curves $\gamma_z$ that do not have a well-defined asymptotic average slope because their finite time horizon slopes oscillate between $0$ and $\infty$.

Using Theorem~\ref{th:no-average-slope}, we can also construct a time-dependent space-time stationary, weak mixing, and smooth one-dimensional positive velocity field $u(t,x)\in\R$, $(t,x)\in\R\times\R$ that does not give rise to
a well-defined asymptotic speed. 

Given a vector field $u$ and a starting point
$(t_0,x_0)$, we define $(x_{(t_0,x_0)}(t))_{t\ge t_0}$ via
\begin{align}
\label{eq:nonauto-sde}
\dot x_{(t_0,x_0)}(t)&=u(t,x_{(t_0,x_0)}(t)),\quad t\ge t_0,\\
\label{eq:nonauto-ini}
x_{(t_0,x_0)}(t_0)&=x_0.
\end{align}

\begin{theorem}
\label{thm:no-average-speed}
If $0\le u_0<u_1<\infty$, then there is a weakly mixing space-time stationary  random vector field $u\in C^{\infty}(\R\times\R)$ and such that $u(t,x)\in [u_0, u_1]$ for all $(t,x)\in\R\times\R$ and, with probability 1, for every $(t_0,x_0)$ the trajectory $x_{(t_0,x_0)}$ solving
\eqref{eq:nonauto-sde} --\eqref{eq:nonauto-ini} satisfies
\begin{equation}
\liminf_{t\to+\infty} \frac{x_{(t_0,x_0)}(t)}{t}=u_0,\qquad \limsup_{t\to+\infty} \frac{x_{(t_0,x_0)}(t)}{t}=u_1.
\end{equation}
\end{theorem}
It is easy to see that for a velocity field $u$ with
all the properties claimed in Theorem~\ref{thm:no-average-speed}, we can take the 
pushforward of the vector field $v$ constructed 
in Theorem~\ref{th:no-average-slope} under the linear map 
defined by the invertible matrix 
\[
 A=\begin{pmatrix}
    1 & 1  \\
    u_0 & u_1  
\end{pmatrix},
\]
i.e., we define
\[
u(t,x)=Av\left(A^{-1}\begin{pmatrix}
t\\x 
\end{pmatrix} \right).
\]

The rest of this note is organized as follows.
We extend the construction in~\cite{2016arXiv161200434C} in two stages.  First, in Section~\ref{sec:vector-field-construction},
we introduce two vector fields on the unit square associated with vertical and horizontal arrows, respectively, 
and then obtain a smooth vector field on $\R^2$ tesselating it by square tiles in agreement 
with the random arrow fields introduced in~\cite{2016arXiv161200434C}. The trajectories generated by the resulting vector field do not have asymptotic directions but the vector field itself lacks stationarity with respect to shifts in $\R^2$, so one cannot even speak about weak mixing. To fix this and finish the proof, a random Poissonian deformation of this vector field is introduced in Section~\ref{sec:weak-mixing}.

{\bf Acknowledgments.} We learned about the question that we study in this paper from Alexei Novikov. We are grateful to him, Leonid Koralov, and Arjun Krishnan for discussions that followed. In addition, we gratefully acknowledge partial support from NSF via Award DMS-1811444.  

\section{Constructing a smooth vector field from an arrow field on $\Z^2$}
\label{sec:vector-field-construction}
Let $r=(1,0)$ and $u=(0,1)$  be the standard coordinate vectors on the plane pointing right and up, respectively.
On $\Z^2$, an (up-right) arrow field is a function $\alpha :\Z^2 \to \{r,u \}$, and the random walk~$X_z:  \N \to \Z^2$ that starts at $z$
and follows the arrow field $\alpha$ is defined by
\begin{equation*}
X_z (0) = z, \quad X_z( n) = X_z( n-1) + \alpha\big(  X_z( n-1) \big).
\end{equation*}
In \cite{2016arXiv161200434C}, the authors constructed an ergodic up-right random walk on $\Z^2$
such that no trajectories have asymptotic directions, and hence by the result therein all random walks must coalesce.
More precisely, they proved the following:
\begin{theorem}
\label{thm:Z2-vector-field}
There is  a $\Z^2$-ergodic dynamical system~$((T_z )_{z \in \Z^2},\Omega, \Fc, \nu)$ 
and a measurable function $\bar{\alpha}: \Omega \to \{r,u \}$ that defines a stationary $\Z^2$-arrow field by 
\begin{equation*}
\alpha^{\omega}(z) = \bar{\alpha}(T_z \omega),  \quad \omega \in \Omega,\quad z\in\Z^2,
\end{equation*}
such that none of the corresponding family of random walks~$(X^{\omega}_z)_{z \in \Z^2}$ have an asymptotic direction and all the random walks ~$(X^{\omega}_z)_{z \in \Z^2}$ coalesce. More precisely,
for $\nu$-a.e.~$\omega \in \Omega$,
\begin{equation}
  \label{eq:arrow-field-no-direction}
  \liminf_{n \to \infty} \frac{X^{\omega}_z( n) \cdot u}{ X^{\omega}_z(n) \cdot r} = 0, \quad
  \limsup_{n \to \infty} \frac{X^{\omega}_z( n) \cdot u}{ X^{\omega}_z( n) \cdot r} = \infty, \qquad z \in \Z^2, 
\end{equation}
and 
\begin{equation}
\label{eq:almost-surely-coalescence}
\forall z_1, z_2 \in \Z^2,\ \exists k_1, k_2 \text{\ \rm such that\ } X_{z_1}^{\omega}(k_1) = X_{z_2}^{\omega}(k_2).
\end{equation}
\end{theorem}
In fact,
 the authors in \cite{2016arXiv161200434C} constructed the $\Z^2$-system as the product of
two appropriately chosen~$\Z$-systems $(S_1, X, \Bc, \lambda)$ and $(S_2,Y, \Bc, \lambda)$, with $X=Y= [0,1)$,
$\Bc$  being the Borel $\sigma$-algebra, and $\lambda$ the Lebesgue measure.   The  product $\Z^2$-action is defined by $T_{(a,b)}(x,y) = (S_1^a x, S_2^by)$.
This~$\Z^2$-system is weakly mixing since both~$\Z$-systems are.
(See Section~\ref{sec:appendix} for a collection of definitions and statements in ergodic theory that will be used in this paper.)

In this section, we will demonstrate how to create a smooth vector field~$\Psi_{\alpha}$ from any given up-right $\Z^2$-arrow field $\alpha$, such that the integral curves of $\Psi_{\alpha}$ have similar
behavior as the random walks following the arrow field $\alpha$.  When~$\alpha$ is given by Theorem~\ref{thm:Z2-vector-field}, $\Psi_{\alpha}$ will satisfy~(\ref{eq:integral-curves-no-direction})
(Theorem~\ref{thm:no-direction-for-toy-model}). 
  
Suppose $V_u$ and $V_r$ are two smooth fixed vector
fields on $[0,1]^2$ roughly behaving like ``up arrow'' and ``right arrow'' that will be specified later.
The vector field $\Psi_{\alpha}$, as a functional of $\alpha$, is defined by piecing together copies of~$V_r$ and~$V_u$:
\begin{equation}
\label{eq:def-of-Psi}
\Psi_{\alpha}(x + i,y + j )=  V_{\alpha(i,j)}(x,y), \quad (i,j) \in  \Z^2, (x,y) \in [0,1)^2.
\end{equation}
Naturally, we assume that $V_u$ and $V_r$ are diagonally symmetric to each other, i.e., 
\begin{equation}\label{eq:diagonal-symmetry}
V^{1}_u(x,y) = V^{2}_r(y,x) , \quad V^{2}_u(x,y) =
V^{1}_r(y,x), \quad (x,y) \in [0,1]^2.
\end{equation}
To simplify the construction, we also require that
that $V_r$ (and hence $V_u$) is itself diagonally symmetric near the boundary, that is, there exists~$\delta > 0$ such that 
\begin{equation}
  \label{eq:symmetric-at-the-boundary}
  V_r(x,y) = V_r(y,x), \quad
  (x,y ) \in \Gamma_{\delta},
\end{equation}
where for $h \ge 0$, $\Gamma_{h}$ is the region 
\begin{equation*}
\Gamma_{h} = \big\{ (x,y) \in [0,1]^2:\,  \min\{x,1-x,y,1-y\} \le h\big\}, \quad h \ge 0.
\end{equation*}

The construction of $V_r$ and $V_u$ is as follows.  Let us take any $\delta < 1/10$.
Let $\tilde{F}_r$ be a potential function in~$[0,1]^2$ as defined in
Fig~\ref{fig:def-of-F-tilde}. The potential $\tilde{F}_r$ is a piece-wise linear function so that $\nabla \tilde{F}_r$ is constant in each polygon region.  At the four pentagon regions at the
  corners $\tilde{F}_r$ is given by the following:
\begin{equation*}
\tilde{F}_r(x,y) = 
\begin{cases}
3(x+y), & (x,y) \text{ at the SW corner},  \\
3(x+y) -1   ,& (x,y) \text{ at the SE and NW corners },  \\
3(x+y) - 2 , & (x,y) \text{ at the NE corner}.
\end{cases}
\end{equation*}
And at the middle non-convex heptagon $\tilde{F}_r(x,y )  = 2x + 1$.  The values of $\tilde{F}_r$ at
all the vertices are then determined, given in boldface, and $\tilde{F}_r$ in the remaining triangle
regions are given by the linear interpolation of its values at the vertex.

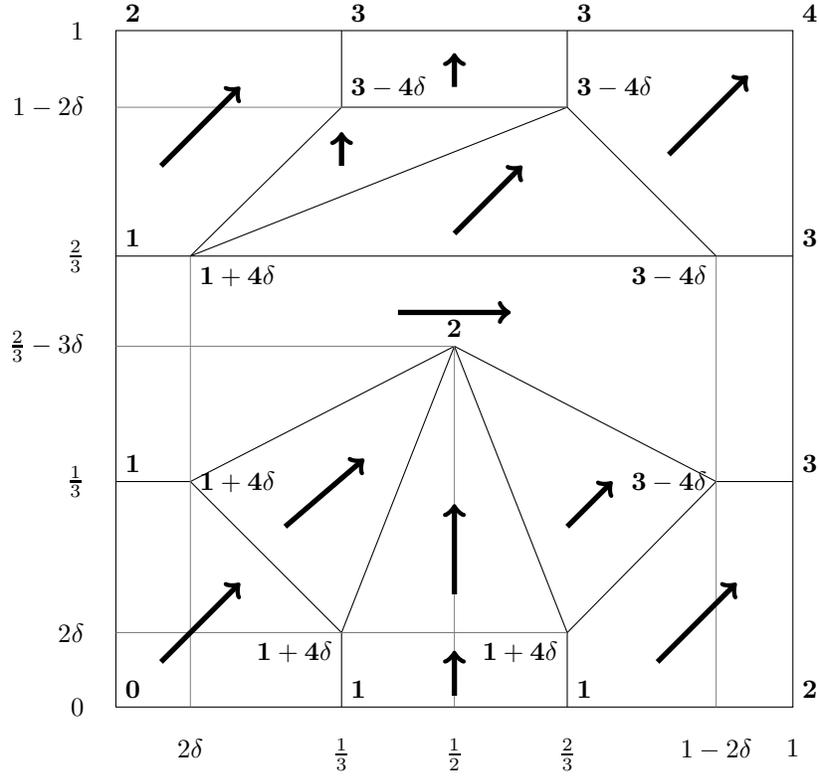
\begin{figure}[h]
  \centering
\begin{tikzpicture}[scale =3]
  \draw (0,0) -- (3,0) -- (3,3) -- (0,3) -- (0,0); 
  \draw (0,2) -- (3,2);
  \draw (0.33, 2) -- (1, 2.66) -- (1,3);
  \draw (0.33, 2) -- (2, 2.66) -- (2,3);
  \draw (2,2.66) -- (2.66, 2);
  \draw (0,1) -- (0.33, 1) -- (1.5, 1.6) -- (2.66, 1) -- (3,1);
  \draw (0.33, 1) -- (1, 0.33) -- (1,0) ;
  \draw (1.5, 1.6) -- (1, 0.33);
  \draw (1.5, 1.6) -- (2, 0.33) -- (2, 0) ;
  \draw (2, 0.33) -- (2.66, 1);
  \draw (1, .33) -- (2, .33);
  \draw (1,2.66) -- (2, 2.66);
  
  \draw[help lines] (0.33, 0) -- (0.33, 2);
  \draw[help lines] (2.66, 0) -- (2.66, 2);
  \draw[help lines] (0, 0.33) -- (2,  0.33);
  \draw[help lines ] (0, 2.66) -- (2, 2.66);
  \draw[help lines ] (0, 1.6) -- (1.5, 1.6) -- (1.5, 0);

  \draw[->, line width =2] (1.25, 1.75) -- (1.75, 1.75);
  \draw[->, line width = 2] (0.2, 0.2) -- (0.55, 0.55);
  \draw[->, line width  = 2] (2.45, 2.45) -- (2.8, 2.8);
  \draw[->, line width = 2 ]  (1.5, 0.05) -- (1.5, 0.25);
  \draw[->, line width = 2] (1.5, 0.5) -- (1.5, 0.9);
  \draw[->, line width = 2] (2.4, 0.2) -- (2.75, 0.55);
  \draw[->, line width = 2] (0.2, 2.4) -- (0.55, 2.75);
 
  \draw[->, line width = 2] (1.5, 2.75) -- (1.5, 2.90);
  \draw[->, line width = 2] (1, 2.4) -- (1, 2.55);
  
  \draw[->, line width =2 ] (0.75, 0.8) -- (1.1, 1.10);
  \draw[->, line width = 2] (2, 0.8) -- (2.2, 1);

  \draw[->, line width = 2] (1.5, 2.1) -- (1.8, 2.4);

  \node[left] at (-0.1, 0) {0};
  \node[left] at (-0.1, 0.33) {$2\delta$};
  \node[left] at (-0.1, 1) {$\frac{1}{3}$};
  \node[left] at (-0.1, 1.6) {$\frac{2}{3} - 3\delta$};
  \node[left] at (-0.1, 2) {$\frac{2}{3}$};
  \node[left] at (-0.1, 2.66) {$1-2\delta$};
  \node[left] at (-0.1, 3) {1};

  \node[below] at (0.33, -0.1) {$2\delta$};  
  \node[below] at (1, -0.1) {$\frac{1}{3}$};
  \node[below] at (1.5, -0.1) {$\frac{1}{2}$};
  \node[below] at (2, -0.1) {$\frac{2}{3}$};
  \node[below] at ( 2.66, -0.1) {$1-2\delta$};
  \node[below] at (3, -0.1) {1};

  \node[above right] at (0, 0) { \textbf{0}};
  \node[above right] at (0, 1) { \textbf{1}};
  \node[above right] at (0, 2) { \textbf{1}};
  \node[above right] at (0, 3) { \textbf{2}};
  \node[above right] at (1,3) { \textbf{3}};
  \node[above right] at (2,3) { \textbf{3}};
  \node[above right] at (3,3) { \textbf{4}};
  \node[above right] at (3,2) { \textbf{3}};
  \node[above right] at (3,1) { \textbf{3}};
  \node[above right] at (3,0) { \textbf{2}};
  \node[above right] at (2,0) { \textbf{1}};
  \node[above right] at (1, 0) { \textbf{1}};

  \node[below right] at (0.33, 2) { $\mathbf{1+4\delta}$};
  \node[below left] at (2.66,  2) { $ \mathbf{3-4\delta}$};
  \node[above right] at (1,2.66) { $\mathbf{3 - 4\delta}$};
  \node[above right] at (2, 2.66) { $\mathbf{3 - 4\delta}$};
  \node[above] at (1.5, 1.6) { \textbf{2}};
  \node[right] at (0.33, 1) {$\mathbf{1+4\delta} $};
  \node[below left] at (1, 0.33) { $\mathbf{1+4\delta}$};
  \node[below left] at (2, 0.33) { $\mathbf{1+4\delta}$};
  \node[left] at (2.66, 1) {$\mathbf{3-4\delta}$};
\end{tikzpicture}
  \caption{Definition of $\tilde{F}_r$ in the unit square $[0,1]^2$.  This potential is continuous on $[0,1]^2$ and linear
  in every polygonal cell. The values 
  of~$\tilde{F}_r$  at the tesselation vertices are given in boldface.  The arrows indicate the direction of $\nabla \tilde{F}_r$.} 
  \label{fig:def-of-F-tilde}
\end{figure}
We extend~$\tilde{F}_r$ to $\R^2$ by 
\begin{equation}\label{eq:def-of-F-tilde}
\tilde{F}_r(x+i,y + j) = \tilde{F}_r(x,y) + 2(i+j),\quad (i,j) \in \Z^2, (x,y) \in [0,1)^2,
\end{equation}
and then by smoothing it we define $F_r = \eta * \tilde{F}_r$, where~$\eta\in C^\infty$ is a radially symmetric 
kernel supported on~$B_0(\delta) = \{ (x,y): x^2+y^2 \le \delta^2\}$. Finally, we define~$V_r$ as the restriction of the gradient field $\nabla F_r$ to $[0,1]^2$:
\begin{equation*}
V_r(x,y) = (\nabla F_r )(x,y), \quad (x,y) \in [0,1]^2.
\end{equation*}
We define $V_u$ through diagonal symmetry~(\ref{eq:diagonal-symmetry}).

\begin{lemma}
  \label{lem:smoothness-of-psi}
  Let $V_r$ and $V_u$ be defined as above.  For any arrow field $\alpha$, the vector field $\Psi_{\alpha}$ as defined in~(\ref{eq:def-of-Psi}) is smooth and bounded.  Moreover, 
  \begin{equation}
    \label{eq:non-degeneracy}
\Psi_{\alpha}^{1} \ge 0, \quad \Psi_{\alpha}^{2} \ge 0, \quad \Psi_{\alpha}^{1} + \Psi_{\alpha}^{2} \ge c > 0,
\end{equation}
for some constant $c$.
\end{lemma}

\begin{proof}
By~(\ref{eq:def-of-F-tilde}), $\nabla \tilde{F}_r$ is $\Z^2$-periodic, i.e., 
\begin{equation*}
\nabla \tilde{F}_r(x+i, y+j) = \nabla \tilde{F}_r(x,y), \quad (i,j) \in \Z^2,
\end{equation*}
Hence~$\nabla F_r = \eta * \nabla \tilde{F}_r$ is also $\Z^2$-periodic.  This implies~$\nabla F_r = \Psi_{\alpha_r}$, where $\alpha_r$ 
is the~$\Z^2$-arrow field with right arrows only.   From the
$\Z^2$-periodicity of $\nabla \tilde{F}_r$
and Fig.~\ref{fig:def-of-F-tilde}, it is also easy to see that 
\begin{equation*}
\nabla \tilde{F}_r(x,y) = \nabla \tilde{F}_r(y,x), \quad (x,y) \in \bar{\Gamma}_{2\delta}, 
\end{equation*}
where 
\begin{equation*}
\bar{\Gamma}_h = \bigcup_{(i,j) \in \Z^2} \{  (x+i, y+j) : (x,y) \in \Gamma_{h} \},\quad h\ge 0.
\end{equation*}
Since the smoothing kernel $\eta$ is supported on $B_0(\delta)$, $\nabla F_r = \eta * \nabla \tilde{F}_r$ will satisfy
\begin{equation*}
\nabla F_r(x,y) = \nabla F_r(y,x), \quad (x,y) \in \bar{\Gamma}_{\delta}.
\end{equation*}
Therefore, $V_r$ satisfies (\ref{eq:symmetric-at-the-boundary}).

Let $\alpha$ be any arrow field.  Due to~(\ref{eq:symmetric-at-the-boundary}), we have~$\Psi_{\alpha} = \Psi_{\alpha_r}$ in~$\bar{\Gamma}_{\delta}$, which implies that $\Psi_{\alpha}$ is smooth in a
neighborhood of $\bar{\Gamma}_0$.
Since, in addition, $V_r$ and $V_u$ are smooth in $(0,1)^2$, $\Psi_{\alpha}$ is smooth everywhere.

Finally, the condition~(\ref{eq:non-degeneracy}) holds for $\Psi$ since it holds for $\nabla \tilde{F}_r$.
\end{proof}

It is also easy to see that we have the following corollary:
\begin{corollary}
  \label{cor:existence-of-potential}
  For any arrow field~$\alpha$, there is a potential~$F_{\alpha}$ such that $\Psi_{\alpha}=\nabla F_\alpha$.
\end{corollary}

\begin{theorem}
  \label{thm:no-direction-for-toy-model}
  Let $\alpha$ be the stationary arrow field introduced in Theorem~\ref{thm:Z2-vector-field} and~$\Psi_{\alpha}$ be the corresponding vector field defined by~(\ref{eq:def-of-Psi}).  Then, with
    probability one, all integral curves $\gamma_z$ of $\Psi_{\alpha}$ will satisfy~(\ref{eq:integral-curves-no-direction}).
\end{theorem}

\begin{proof}
  By Lemma~\ref{lem:smoothness-of-psi}, $\Psi_{\alpha}$ is smooth, bounded and nondegenerate, so the integral curves of $\Psi_{\alpha}$ are well-defined.

  We can partition $\R^2$ into the union of unit squares:
\begin{equation*}
  \R^2 = \bigcup_{(i,j) \in \Z^2} S_{(i,j)}, \quad S_{(i,j)} = [i,i+1) \times [j,j+1).
\end{equation*}
 We say that~$z \in S_{(i,j)}$ is regular, if the curve~$\gamma_z$ visit these squares in the order given by the random walks $X_{(i,j)}$.
  It suffices to show that with probability one,  every curve of~$\Psi_{\alpha}$ passes through a regular point.  The conclusion of the theorem follows from~(\ref{eq:arrow-field-no-direction}).

We notice that~$V_r(x,y) \equiv (2,0)$ in the strip 
\begin{equation*}
\{ (x,y):\ 0 \le x \le 1,\ 2/3 - 2\delta \le y \le 2/3 - \delta \}.
\end{equation*}
This follows from the fact that~$\nabla \tilde{F}_r \equiv (2,0)$ in the strip
\begin{equation*}
\{  (x,y) : x \in \R, 2/3 - 3\delta \le y \le 2/3 \}
\end{equation*}
and that~$\eta$ is a kernel supported on~$B_0(\delta)$.  Therefore, all the integral curves of $V_r$ entering the unit square
through the set
\begin{equation*}
s_1 = \{  (0, y): 0 \le y \le 2/3 - \delta \} \cup \{ (x,0) : 0 \le x \le 1 \}
\end{equation*}
have to exit through 
\begin{equation*}
s_2 = \{ (1,y): 0 \le y \le 2/3 - \delta \}.
\end{equation*}
Let us define $\Omega_{(i,j)} \subset S_{(i,j)}$ to be 
\begin{equation*}
\Omega_{(i,j)} = 
\begin{cases}
\{ (x,y): i \le x < i+1, j \le y \le j+ 2/3 - \delta \}, & \alpha(i,j) = r,  \\
\{ (x,y): i \le  x \le i + 2/3 - \delta, j \le y < j+1\},   & \alpha(i,j) = u.  \\
\end{cases}
\end{equation*}

We now claim that any point in $\Omega = \bigcup_{(i,j) \in \Z^2} \Omega_{(i,j)}$ is regular.

Suppose $(i_0, j_0) \in \Z^2$ and  $z \in
\Omega_{(i_0, j_{0})}$.  If $\alpha(i_0, j_0) = r$, then our construction implies that 
after exiting $S(i_0, j_0)$, $\gamma_z$ enters $\Omega_{(i_0+1, j_0)}\subset S_{(i_0+1, j_0)}$.
If $\alpha(i_0, j_0) = u$, then
after exiting $S(i_0, j_0)$, $\gamma_z$ enters $\Omega_{(i_0, j_0+1)}\subset S_{(i_0, j_0+1)}$, see
 Fig.~\ref{fig:2}. Applying these steps inductively, we see that~$\gamma_z$ indeed~``follows the arrows'', so $z$ is regular.  This proves the claim.
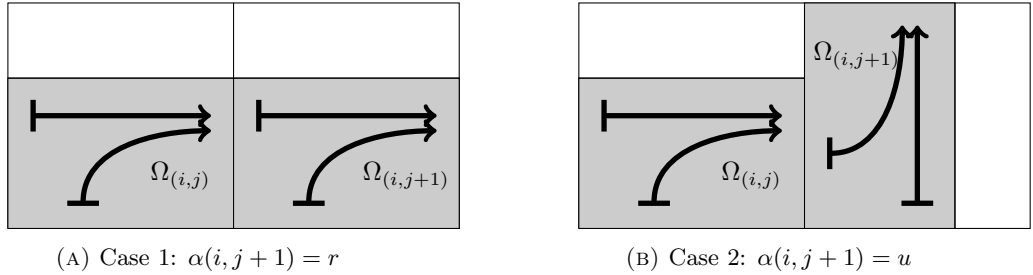
\begin{figure}[h]
  \centering
  \begin{subfigure}{0.4\textwidth}
  \begin{tikzpicture}
    \draw (0,0) -- (0,3) -- (3,3) -- (3,0) -- (0,0);
    \draw [fill = black!20!white] (0,0) -- (0,2) -- (3,2) -- (3,0) -- cycle;
    \draw (0, 2) -- (3,2) ;
    \draw[|->, line width = 2] (0.3, 1.5) -- (2.7, 1.5);
    \draw[|-> ,line width = 2] (1, 0.3) to [out = 90, in = 180] (2.7, 1.3);
    \node at (2.3, 0.7) {$\Omega_{(i,j)}$};
    
    \draw (3,0) -- (6, 0) -- (6,3) -- (3,3);
    \draw [fill = black!20!white] (3,0) -- (3, 2) -- (6,2) -- (6,0) -- cycle;    
    \draw (3, 2) -- (6,2) ;
    \draw[|->, line width = 2] (3.3, 1.5) -- (5.7, 1.5);
    \draw[|-> ,line width = 2] (4, 0.3) to [out = 90, in = 180] (2.7+3, 1.3);
    \node at (2.3 + 3, 0.7) {$\Omega_{(i,j+1)}$};    
  \end{tikzpicture}
  \caption{Case 1: $\alpha(i,j+1) = r$}
\end{subfigure}
\hfill
\begin{subfigure}{0.4\textwidth}
\begin{tikzpicture}
    \draw (0,0) -- (0,3) -- (3,3) -- (3,0) -- (0,0);
    \draw [fill = black!20!white] (0,0) -- (0,2) -- (3,2) -- (3,0) -- cycle;
    \draw (0, 2) -- (3,2) ;
    \draw[|->, line width = 2] (0.3, 1.5) -- (2.7, 1.5);
    \draw[|-> ,line width = 2] (1, 0.3) to [out = 90, in = 180] (2.7, 1.3);
    \node at (2.3, 0.7) {$\Omega_{(i,j)}$};
    
    \draw (3,0) -- (6, 0) -- (6,3) -- (3,3);
    \draw [fill = black!20!white] (3,0) -- (3,3) -- (5,3) -- (5,0) -- cycle;    
    \draw (5,3) -- (5,0) ;
    \draw[|->, line width = 2] (3 + 1.5 ,  0.3) -- (3 + 1.5, 2.7);
    \draw[|-> ,line width = 2] (3+0.3, 1) to [out = 0, in = 270] (3+1.3, 2.7);
    \node at (0.7 + 3, 2.3) {$\Omega_{(i,j+1)}$};
    
  \end{tikzpicture}
    \caption{Case 2: $\alpha(i,j+1) = u$}
\end{subfigure}
  \caption{Illustration of the flow when $\alpha(i,j) = r$.}
  \label{fig:2}
\end{figure}

Furthermore, since all walks coalesce due to 
Theorem~\ref{thm:Z2-vector-field}, any up-right curve (i.e., $\gamma(t)$ such that $\gamma'(t) \cdot
r \ge 0, \gamma'(t) \cdot u \ge 0, \gamma'(t) \cdot (r+u) > 0$) must intersect~$\Omega$.  This
implies that any integral curve of $\Psi_{\alpha}$ passes through some regular point.
The proof is complete.
\end{proof}

\section{Weakly mixing vector field}
\label{sec:weak-mixing}
The vector field ~$\Psi_{\alpha}$ constructed in the previous section has all the properties that are required in Theorem~\ref{th:no-average-slope} except $\R^2$-stationarity and weak mixing,
although its distribution is
invariant under $\Z^2$-shifts. The goal of this section is to
modify the vector field and gain those properties.

To obtain an $\R^2$-stationary and ergodic random vector field without requiring the weak mixing property, we could introduce a simple randomization by adding an independent  $[0,1]^2$-uniformly distributed random shift to~$\Psi_{\alpha}$. To obtain a weakly mixing vector field we need to apply an additional random deformation that we proceed to describe.

Let $\mu = \sum_i \delta_{a_i}$ and $\nu = \sum_j \delta_{b_j}$ be two  Poissonian point
processes on~$\R$ and fix a family of
positive~$C^{\infty}$-functions $(\phi_{\Delta})_{\Delta > 0}$ with the following properties:
\begin{enumerate}[1.]
\item $\phi_{\Delta}(x) \equiv 1$ near $x=0$ and $x = \Delta$,
\item $\int_0^{\Delta}\phi_{\Delta}(x) \, dx = 1$,
\item $(\Delta, x) \mapsto
\phi_{\Delta}(x)$ is continuous (and hence measurable).
\end{enumerate}
We define
\begin{multline*}
  \varphi_{\mu,\nu}(x,y) =
  \Big( \mu( (0,x])  + \int_0^{x-\underline{a}}  \phi_{\bar{a} - \underline{a}} (t) \, dt,\
  \nu( (0,y])  + \int_0^{ y - \underline{b}}
\phi_{\bar{b} - \underline{b}} (t) \, dt  
  \Big),
\end{multline*}
where 
\begin{align*}
\bar{a} &= \bar{a}(x) = \inf \{ a_i: a_i < x \}, & \underline{a} &= \underline{a}(x) = \sup \{
a_i : a_i \le x \},\\
\bar{b} &= \bar{b}(y) = \inf \{ b_j: b_j < x \}, & \underline{b} &= \underline{b}(y) = \sup \{
b_j :b_j \le x \},
\end{align*}
and $\mu((0,x])$ (resp.\ $\nu((0,y])$) is the number of Poissonian points in the interval $(0,x]$
(resp.\ $(0,y]$), with a ``$-$'' sign if $x < 0$ (resp.\ $y < 0$).
If we order the Poisson points in the following way:
\begin{equation*}
  a: \cdots < a_{-1} < a_0 \le 0 < a_1 < \cdots, \quad
  b:  \cdots < b_{-1} < b_0 \le 0 < b_1 < \cdots,
\end{equation*}
then $\phi_{\mu,\nu}$ is a~$C^{\infty}$-automorphism of $\R^2$ and satisfies
  \begin{equation}\label{eq:image-of-phi-ab}
\varphi_{\mu,\nu} ( \{ x=a_i \}) = \{ x=i \}, \quad \varphi_{\mu,\nu}(\{  y = b_j \}) = \{ y = j \}, \qquad i, j \in \Z.
\end{equation}
In particular, $\varphi_{\mu,\nu}$ maps the rectangle~$R_{(i,j)} = [a_i, a_{i+1}) \times [b_j , b_{j+1}) $ to the unit square~$S_{(i,j)}$.

Let us consider  the pushforward  of $\Psi_{\alpha}$ under the map~$\varphi^{-1}$, i.e., the vector field
\begin{equation*}
\Phi(\mathbf{x}) = D \varphi_{\mu,\nu}^{-1} \big( \varphi(\mathbf{x}) \big) \cdot
\Psi_{\alpha}\big(\varphi_{\mu,\nu}(\mathbf{x}) \big)
= \Big(  D \varphi_{\mu,\nu}(\mathbf{x}) \Big)^{-1} \Psi_{\alpha} \big(  \varphi_{\mu,\nu}(\mathbf{x}) \big), \quad \mathbf{x}
\in \R^2,
\end{equation*}
where~$D f$ denotes the Jacobian matrix of~$f$ and $\Psi_{\alpha}$ is introduced in section~\ref{sec:vector-field-construction}.   
Due to~(\ref{eq:image-of-phi-ab}), in each rectangle~$R_{(i,j)}$, the vector field~$\Phi$ is a ``deformation'' of
either~$V_r$ or~$V_u$, depending on whether~$\alpha(i,j) = u$ or~$r$.

We will show that if $\alpha$, $\mu$ and $\nu$ are independent, then~$\Phi$ is stationary and weakly mixing. We start by a formal construction of an appropriate $\R^2$-system.
Let~$((L_v)_{v \in \R}, \mathcal{M}, \mathrm{P}_{\mathcal{M}})$ be a~$\R^1$-system where~$\mathcal{M}$ is the space of locally finite configurations of points on $\R$ (they can be identified with integer-valued measures such that masses of all atoms equal~$1$), $\mathrm{P}_{\mathcal{M}}$ is the Poisson measure
on~$\mathcal{M}$ with intensity~$1$, and the~$\R^1$-action~$L_v$ acting on~$\mu = \sum \delta_{a_i}$
by~$L_v \mu = \sum \delta_{a_i -  v}$. 
We also recall the $\Z^1$-systems~$(S_1, X, \lambda)$ and~$(S_2,Y, \lambda)$ from Section~\ref{sec:vector-field-construction}.
Let us consider the following skew-products
\begin{equation}
\label{eq:skew-product-1}
((L_v)_{v \in \R}, \mathcal{M} \times X, \mathrm{P}_{\mathcal{M}} \otimes \lambda), \qquad
L_v (\mu, x) = (L_v \mu, S_1^{\mu ((0,v])} x),
\end{equation}
and
\begin{equation}
\label{eq:skew-product-2}
((L_v)_{v \in \R}, \mathcal{M} \times Y, \mathrm{P}_{\mathcal{M}} \otimes \lambda), \qquad
L_v (\nu, y) = (L_v \nu, S_2^{\nu ((0,y])} y).
\end{equation}

Let us take the product of~(\ref{eq:skew-product-1}) and~(\ref{eq:skew-product-2}):
\begin{equation}\label{eq:R2-system}
  \big( ( \hat{L}_{v,w})_{(v,w) \in \R^2}, \hat{\Omega}, \Pp \big)
  = \big(  ( L_v \times L_w )_{(v,w) \in \R^2},  \mathcal{M}^2 \times X \times Y,
  \mathrm{P}_{\mathcal{M}}^{2} \otimes \lambda^{2}).
\end{equation}
For $\hat{\Omega} \ni \hat{\omega} = (\mu,  \nu, x, y)$, one can check that the vector field $\Phi$ satisfies
\begin{equation}\label{eq:def-of-the-deformed-vec-field}
\Phi^{\hat{\omega}}(v,w) = \Big(  D \varphi_{\mu,\nu}(v,w) \Big)^{-1} \Psi_{\alpha(x,y)} \big(
\varphi_{\mu,\nu}(v,w) \big) =  \hat{\alpha}(\hat{L}_{v,w}\  \hat{\omega}) ,
  %\Phi^{\hat{\omega}}(v,w) \equiv \nabla (F_{\alpha(x,y)} \circ \varphi_{\mu,\nu})(v,w) = \hat{\alpha}(\hat{L}_{v,w}\  \hat{\omega}) ,
\end{equation}
where
\begin{equation*}
  \hat{\alpha}(\mu,\nu, x, y) = \Big( D \varphi_{\mu,\nu}(0,0) \Big)^{-1}   V_{\bar{\alpha}(x,y)} (\varphi_{\mu,\nu}(0,0)).
\end{equation*}
 The definition~(\ref{eq:def-of-the-deformed-vec-field}) implies that $\Phi$ is stationary.  The
 following theorem states that it is weakly mixing.
\begin{theorem}
  \label{thm:weak-mixing-vector-field}
  The $\R^2$-system~(\ref{eq:R2-system}) is weakly mixing.  Moreover, with probability one, all integral curves of the vector field $\Phi^{\hat{\omega}}$ satisfy~(\ref{eq:integral-curves-no-direction}).
\end{theorem}

The fact that~(\ref{eq:R2-system}) is weakly mixing is implied by the following and Theorem~\ref{thm:product-weak-mixing}.
\begin{lemma}
  \label{lem:R1-system-is-weak-mixing}
The $\R^1$-systems~(\ref{eq:skew-product-1}) and~(\ref{eq:skew-product-2}) are weakly mixing.
\end{lemma}

\begin{proof}
We will only show that~(\ref{eq:skew-product-1}) is weakly mixing.  By
Definition~\ref{def:weak-mixing}, this is equivalent to the ergodicity of its direct product with
itself, i.e., the~$\R^1$-system
\begin{equation}
\label{eq:R1-system}
( ( L^{2 }_v )_{v \in \R} ,  \mathcal{M}^2 \times X^2, \mathrm{P}_{\mathcal{M}}^{2} \otimes \lambda^2).
\end{equation}

For~$(\mu,\mu', x, x') \in \mathcal{M}^2 \times X^2$, let us write~$L^2_v(\mu, \mu', x, x' ) = (\mu_v, \mu'_v, x_v, x'_v)$. 
We notice that under the measure~$\mathrm{P}_{\mathcal{M}}^{2 } \times \lambda^2$, $ (x_v,
x_v')_{v \in \R}$ is a Markov jump process on~$X^2$ starting from $\lambda^2$, jumping from~$(x,x')$ to~$(x,S_1x')$ with rate 1 at times
recorded by~$\mu'$ and from~$(x,x')$ to~$(S_1x,x')$ with rate 1 at times recorded by~$\mu$.
The $\R^1$-action~$L_v^2$ acting on~$\mathcal{M}^2 \times X^2$ is the time shift of this Markov process.

Therefore,
the ergodicity of~(\ref{eq:R1-system})
% ~$\R^1$-system~$( \{ L^{2 }_v \}_{v \in \R} , \tilde{X}^2, \tilde{P}^{\otimes 2})$ 
is equivalent to the ergodicity
of the stationary Markov process~$(x_v, x'_v)_{v \in \R}$. The ergodicity of a stationary Markov process
can be described in terms of the associated semigroup and invariant measure.
We recall that for a Markov semigroup $\Prm=(\Prm_t)_{t\ge0}$  and a $\Prm$-invariant measure~$\nu$ (i.e., satisfying $\nu \Prm^t=\nu$ for all $t\ge 0$), a set~$A$ is called (almost) $\Prm$-invariant if for all $t$, $\Prm^t \ONE_A = \ONE_A$ $\nu$-a.s. The pair
$(\Prm,\nu)$ is ergodic if and only if $\nu(A) = 0$ or $1$ for all invariant sets~$A$.

Suppose that~$A \subset X^2$ is an invariant set for the Markov semigroup $\Prm$ associated with the process $(x_v, x'_v)_{v \in \R}$.
Then, for any $t>0$, 
\begin{equation*}
\Prm^t \ONE_A (x, x') = \sum_{a,b=0}^{\infty} p_t(a,b)\ONE_A(S_1^a x, S_1^bx'),
\end{equation*}
where $p_t(a,b)$ is the probability that the two independent rate $1$ Poisson processes make $a$ and $b$ jumps respectively between times $0$ and $t$. This implies that $A$ is an invariant set for the $\Z^2$-system
\begin{equation*}
%\label{eq:Z2-system}
((S_1^a \times S_1^b )_{(a,b)  \in \Z^2}, X^2, \lambda^2).
\end{equation*}
By Theorem~\ref{thm:product-ergodic}, since $(S_1,X)$ is ergodic, this product system is also
ergodic.  This implies that~$\lambda^2(A) = 0$ or $1$ and completes the proof.
\end{proof}

\begin{proof}[Proof of Theorem~\ref{thm:weak-mixing-vector-field}] 
The weak mixing follows from Definition~\ref{def:weak-mixing} and
Lemma~\ref{lem:R1-system-is-weak-mixing}. Since all integral curves of $\Phi$ are
  images of those of $\Psi_{\alpha}$ under the map~$\varphi_{\mu,\nu}^{-1}$,
  (\ref{eq:integral-curves-no-direction}) follows from Theorem~\ref{thm:no-direction-for-toy-model}
and SLLN for \iid exponential random variables.
\end{proof}

\section{Appendix}\label{sec:appendix}
Here we give some standard definitions and facts from the ergodic theory.

Let $G$ be a group. 
We call~$((T_g)_{g \in G}, X, \mathcal{B}, \mu)$ a~$G$-\textit{system}~
if $(T_g)_{g \in G}$ is a measure preserving action of the group~$G$ on a probability space space~$(X,
\mathcal{B}, \mu)$.
When~${G=\Z}$, we will write $(S, X,\mathcal{B}, \mu)$ where $S
= T_1$.  We may omit the~$\sigma$-algebra~$\mathcal{B}$ along with the measure~$\mu$ if the context is clear.

The \textit{product} of two systems, $((T_g)_{g \in G}, X, \mathcal{B}, \mu)$ and~$((T'_h)_{h \in H}, Y, \mathcal{B}',\nu)$, is a~$(G \times H)$-system $((T_g \times T'_h)_{(g,h) \in G\times H},
X \times Y, \mathcal{B} \otimes \mathcal{B}' , \mu \otimes \nu)$.  The group action is defined by 
\begin{equation}
\label{eq:product_action}
(T_g \times T'_h) (x,y ) = (T_gx, T'_hy), \quad g \in G,\ h \in H.
\end{equation}

The \textit{direct product} of two $G$-systems $((T_g )_{g \in G}, X, \mathcal{B}, \mu)$ and 
$((T'_{g'})_{g' \in G}, Y, \mathcal{B}', \nu)$ is again a $G$-system $((T_g \times T'_g)_{g \in G}, X \times Y, \mathcal{B} \otimes \mathcal{B}', \mu \otimes \nu)$, where
$T_g\times T'_g$ is defined according to~\eqref{eq:product_action} with $h=g\in G$, so
this is the diagonal group action of~$G$ on~$X\times Y$.

In the rest of the section and in the paper, the group we are dealing with will always be $\R^d$
or~$\Z^d$, $d\in\N$.  For $g = (g_1, ..., g_d)\in G$, 
$|g| = \max\limits_{1\le i \le d } |g_i|$  its $L^\infty$-norm.
We use $dg$ to denote the Haar measure, i.e., the Lebesgue measure if $G = \R^d$ and counting measure if $G = \Z^d$.

The following are standard definitions on ergodicity and weak mixing for  group actions (see \cite{Bergelson-Gorodnik}). 
\begin{definition}
  \label{def:ergodicity}
  We say that a $G$-system $((T_g )_{g \in G}, X, \mathcal{B}, \mu)$ is ergodic if and only if one of the following
  equivalent conditions holds true: 
\begin{enumerate}[1)]
\item If a set $A$ is almost $G$-invariant, i.e., $\mu(A \Delta T_g A) = 0$ for all $g \in G$, then~ ${\mu(A) = 0}$ or ${\mu(A) =1}$.
\item For any bounded measurable function $f$, 
\begin{equation}\label{eq:def-of-ergodicity-1}
\lim_{R \to \infty} \frac{1}{(2R)^d} \int_{|g| \le R} f(T_g x) \, dg = \int f(x)\,  \mu (dx), \quad
\text{$\mu$-a.s.\ } x.
\end{equation}

\end{enumerate}
\end{definition}

\begin{definition}
\label{def:weak-mixing}
We say that a $G$-system $((T_g)_{g \in G}, X, \mathcal{B}, \mu)$ is weakly mixing if and only if one of the
following equivalent conditions holds true: 
\begin{enumerate}[1)]
\item For any two sets $A$ and $B$, 
\begin{equation*}
\lim_{R \to \infty } \frac{1}{(2R)^d} \int_{|g| \le R} |\mu(T_g A \cap B) - \mu(A) \mu(B)| \,
dg = 0. 
\end{equation*}
\item The direct product~$((T_g \times T_g )_{g \in G}, X \times X)$ is ergodic.
\end{enumerate}
\end{definition}

\begin{theorem}
\label{thm:product-ergodic}
The  product of two ergodic systems is ergodic.
\end{theorem}
\begin{proof}
Let~$(( T_g )_{g \in G}, X, \mathcal{B}, \mu)$ and~$((T'_h)_{h \in H}, Y, \mathcal{B}',\nu)$ be
two ergodic systems.
It suffices to show that~(\ref{eq:def-of-ergodicity-1}) holds true for the product system with~${f(x,y)
= \ONE_{A \times B}(x,y)}$ for any $A \in \mathcal{B}$ and $B \in \mathcal{B}'$.

We can use the ergodicity of $((T_g)_{g \in G}, X)$ and~$((T'_h)_{h \in H}$ to see that
\begin{align*}
&\lim_{R \to \infty}\frac{1}{(2R)^{2d}}\int_{|(g,h)| \le R} \ONE_{A\times B} (T_g x,T'_hy)\, dg\, dh\\
= &\lim_{R \to \infty}\Bigg(\frac{1}{(2R)^d}\int_{|g| \le R} \ONE_A (T_gx)\,  dg\ \cdot \frac{1}{(2R)^d}\int_{|h| \le R}
            \ONE_B(T'_hy)\, dh \Bigg)            
%\nonumber  &=\lim_{R \to \infty}\int_{|g| \le R} \ONE_A (T_g x) dg \cdot \lim_{R \to \infty}\int_{|h| \le R}\ONE_B(T_h 'y) dh \\ 
= \mu(A) \nu(B)
\end{align*}
holds for $\mu$-a.e.\ $x$ and $\nu$-a.e.\ $y$, i.e., 
for $\mu\times
\nu$-a.e.\ $(x,y)$.
The proof is complete.
\end{proof}
\begin{theorem}
\label{thm:product-weak-mixing}
The  product of two weakly mixing systems is weakly mixing.
\end{theorem}

\begin{proof}
Let $((T_g )_{g \in G} , X)$ and $((T'_h)_{h \in H}, Y)$ be two weakly mixing systems.  Their
 product $( ( T_g \times T'_h)_{(g,h) \in G\times H }, X \times Y)$ is weakly mixing
if and only if 
\begin{equation}\label{eq:complicate-system}
(( (T_g \times T'_h) \times (T_g \times T'_h))_{(g,h) \in G \times H} , (X \times Y)
\times (X \times Y) )
\end{equation}
is ergodic.  The latter is isomorphic to the  product of $(( T_g \times T_g )_{g
  \in G}, X \times X)$ and $( (T'_h \times T'_h)_{h \in H} , Y \times Y)$, and both of these
  systems 
are ergodic.  So~(\ref{eq:complicate-system}) is ergodic by Theorem~\ref{thm:product-ergodic}
and this completes the proof.
\end{proof}

\bibliographystyle{alpha}
\bibliography{Burgers}
\end{document}